\documentclass[11pt]{article}
\usepackage{amssymb,amsmath}
\usepackage{mathrsfs}
\usepackage{mathrsfs}
\usepackage{amsfonts}
\usepackage{mathrsfs}
\usepackage{graphicx}
\usepackage{bbm}
\usepackage{amsmath,amsthm,amssymb,amscd}
\usepackage{latexsym}
\usepackage[numbers,sort&compress]{natbib}
\usepackage[all]{xy}
\usepackage{amssymb}

\newtheorem{lemma}{Lemma}[section]
\newtheorem{theorem}{Theorem}[section]

\DeclareMathOperator{\SL}{\mathrm{SL}}

\begin{document}
\date{}

\title{{The Hermite ring conjecture and special linear groups for valuation rings}
{\thanks{This research is supported by the National Science Foundation of China(11501192, 11471108) and Scientific Research Fund of Hunan province education Department(2017JJ3084).
$+$Corresponding author: Dongmei Li, dmli@hnust.edu.cn}}
{\author{\small  Jinwang Liu\quad Dongmei Li$^{+}$\quad  Licui Zheng\\
\small \, \, School of Mathematics and Computing Sciences, Hunan
University\\  \small of   Science and Technology,
 Xiangtan, Hunan, China, 411201}}}

\maketitle
{\bf ABSTRACT}
 \quad In this paper we prove that the Hermite ring conjecture holds for valuation rings $V$, and the special liner group $\SL_n(V[x])$ coincides with the group generated by elementary matrices $E_n(V[x])$ for $n\geq3$. For any arithmetical ring $R$ and $n\geq3$, we show $\SL_n(R[x]) = \SL_n(R)\cdot E_n(R[x])$.

 \vspace{0.2cm}

\section{Introduction}
\normalsize
Serre's conjecture, proposed by J. P. Serre in 1955, asserts that finitely generated projective modules of polynomial rings with finitely many variables over a field, are actually free. In 1976, Quillen \cite{Quil} and Suslin \cite{Sus76} independently gave an affirmative answer to this famous conjecture via completely different approaches. Note that Serre's conjecture is equivalent to saying that any unimodular row $\omega$ in $K^{1\times n}[x_1,\cdots,x_m]$ can be completed to an invertible square matrix, where $K$ is a field. In \cite{Sus77} Suslin established the following $K_1$-analogue of Serre's conjecture (Suslin's stability Theorem): If $R$ is a commutative Noetherian ring and $n\geq \max(3, \dim(R)+2)$ then $\SL_n(R[x_1,\cdots,x_m]) = E_n(R[x_1,\cdots,x_m])$.

Recall that a ring $R$ is called a Hermite ring if every finite generated stably free module is free. It is fairly easy to see that $R$ is Hermite if any unimodular row $\omega \in R^{1\times n}$ can be completed to an invertible square matrix over $R$. In \cite{Lam06,Lam78}, Lam presented the following Hermite ring conjecture: If $R$ is a Hermite ring, then $R[x]$ is Hermite as well. This conjecture seems rather untractable since  there does not seem to exist much evidence for its truth, and so far it has only been verified for some special cases. For instances,  Quillen-Suslin Theorem implies that the Hermite ring conjecture is true for $K[x_1,x_2,\cdots,x_n]$, where $K$ is a field, and Yengui \cite{Yen1} proved that it also holds for rings of Krull dimension $\leq1$.

The following conjecture about $K_1$-analogue question has been proposed in \cite{Yen1}.

\vspace{0.2cm}
{\bf Conjecture 1 }\quad Suppose that $R$ is a ring of Krull dimension $\leq 1$, and $n\geq3$. Then every matrix $\sigma(x) \in \SL_n(R[x])$ is congruent to $\sigma(0)$ modulo $E_n(R[x])$. That is, for every $\sigma(x)\in \SL_n(R[x])$, there is $E \in E_n(R[x])$ such that $\sigma(x) \cdot E = \sigma(0)$.
\vspace{0.2cm}

In fact, by virtue of local-global principle for elementary matrices [5], Conjecture 1 is equivalent to the following Conjecture 2.
\vspace{0.2cm}

{\bf Conjecture 2 }\quad Suppose $R$ is a local ring of Krull dimension $\leq 1$, and
$$M=\left(
      \begin{array}{ccc}
        p & q & 0 \\
        r & s & 0 \\
        0 & 0 & 1 \\
      \end{array}
    \right)\in SL_3(R[x])
$$
Then $M\in E_3(R[x])$.
\vspace{0.2cm}

Liu [7] has proved that Conjectures 1 and 2 are true for valuation rings of Krull dimension $\leq1$. Since it is well-known that a valuation ring is Hermite, it is natural for us to ask the following three question:

(1) Is the Hermite ring conjecture true for valuation rings?

(2) Are conjectures 1 and 2 true for any valuation ring $V$? Furthermore, do we have $\SL_n(V[x]) = E_n(V[x])$ for any $n \geq 3$?

(3) Is conjecture 1 true for arithmetical rings?

In this paper, we shall surround the Hermite ring conjecture, investigate and solve the three problems above.

\section{Notations}

All rings considered in this paper are unitary and commutative. The undefined terminology is standard as in \cite{Lam06}. Recall that a commutative ring $V$ is said to be a \emph{valuation ring} if principal ideals of $V$ are totally ordered by inclusion; that is, any two elements of $V$ are comparable under division. Obviously, a valuation ring is a local ring. A commutative ring $R$ is said to be \emph{arithmetical} if the localization $R_m$ of $R$ at $m$ is a valuation ring for every maximal ideal $m$ of $R$. For any commutative ring $R$ and $n \ge 3$,  $Um_n(R)$ denotes the set of length $n$ unimodular rows over $R$. Let $\SL_n(R)$ denote the special linear group, and $E_n(R)$ denotes the subgroup of $\SL_n(R)$ generated by elementary matrices of the form $I_n+a{\bf e}_{ij}$, where $I_n$ is the identity matrix, $a \in R$, $i \ne j$, and ${\bf e}_{ij}$ is the matrix whose only nonzero entry is the $1_R$ at position $(i, j)$. We know that $E_n(R)$ is a normal subgroup of $\SL_n(R)$.

\section{The Hermite ring conjecture for valuation rings}

In this section, let $R$ be an unitary and commutative ring, and $V$ be a valuation ring. For $f(x)\in V[x]$, we say that $f(x)$ is primitive if at least one coefficient in $f(x)$ is a unit.

We first introduce several necessary lemmas.

\begin{lemma} $([4])$
Let $G=E_n(R)$, and $\gamma=(c_1,c_2,\cdots,c_n)\in Um_n(R)$, $n\geq3$, $b_1,b_2\in R$, such that $c_1b_1+c_2b_2$ is a unit modulo $c_3R+\cdots+c_nR$. Then $$(c_1,c_2,c_3,\cdots,c_n)\sim_G(b_1,b_2,c_3,\cdots,c_n)$$
\end{lemma}

\begin{lemma}
Let $G=E_n(R)$, and $\gamma=(c_1,c_2,\cdots,c_n)\in Um_n(R)$, $n\geq3$. If $u\in R$ is a unit modulo $c_3R+\cdots+c_nR$, then $$(c_1,c_2,\cdots,c_n)\sim_G(uc_1,uc_2,c_3,\cdots,c_n)$$
\end{lemma}

\begin{proof}
Because $(c_1,c_2,\cdots,c_n)\in Um_n(R)$, $u\in R$ is a unit modulo $c_3R+\cdots+c_nR$, then there exist $x_1,x_2,\cdots,x_n,v,y_1,y_2,\cdots,y_n\in R$ such that
\begin{equation*}
\begin{split}
&c_1x_1+c_2x_2+c_3x_3+\cdots+c_nx_n=1,\\
&vu+c_3y_3+\cdots+c_ny_n=1
\end{split}
\end{equation*}
Then
 \begin{equation}\label{1}
   \begin{split}
v((uc_1)x_1+(uc_2)x_2)& = vu(c_1x_1+c_2x_2)\\
  &=(1-c_3y_3-\cdots-c_ny_n)\cdot(1-c_3x_3-\cdots-c_ny_n)\\
  &\equiv1 \mod\ (c_3R+\cdots+c_nR) \end{split}
\end{equation}

From the equation above, we see that $c_1x_1+c_2x_2$ and $(uc_1)x_1+(uc_2)x_2$ are units modulo $c_3R+\cdots+c_nR$. Applying Lemma 3.1 twice, we get
$$\gamma\sim_G(x_1,x_2,c_3,\cdots,c_n)\sim_G(uc_1,uc_2,c_3,\cdots,c_n).$$
\end{proof}

For any unitary commutative ring R, we see that $R[x]$ is also a unitary commutative ring, so the two lemmas above are also true for $R[x]$.

\begin{lemma}
Let $e(x)\in R[x]$ such that $\deg e(x)\geq1$ and $e(0)$ is a unit. Then we have:
\begin{enumerate}
\item  $x$ is a unit modulo $e(x)R[x]$;

\item if $(x^{s_1}f_1(x), x^{s_2}f_2(x), e(x)) \in Um_3(R[x])$, then
\begin{equation*}
(x^{s_1}f_1(x), x^{s_2}f_2(x), e(x)) \sim_{E_3(R[x])} (x^{s_1-k}f_1(x), x^{s_2-k}f_2(x), e(x)),
\end{equation*}
where $s_1,\ s_2,\ k$ are positive integers, and $s_1, s_2 \geq k$;

\item if $(f_1(x), f_2(x), e(x)) \in Um_3(V[x])$, then
\begin{equation*}
(f_1(x), f_2(x), e(x)) \sim_{E_3(R[x])} (x^kf_1(x), x^kf_2(x), e(x)),
\end{equation*}
where $k$ is a positive integer.
\end{enumerate}
\end{lemma}

\begin{proof}
Let $e(x)=b_nx^n + \cdots + b_1x + b_0$ such that $b_n\neq0$ and $b_0$ is a unit, and $g_1(x) = (e(x) - b_0)/x$. Then $x(-b_0^{-1}g_1(x))+b_0^{-1}e(x)=1$, so $x$ is a unit modulo $e(x)R[x]$. Clearly, $(x^{s_1-k}f_1(x), x^{s_2-k}f_2(x), e(x)) \in Um_3(R[x])$. By lemma 3.2, the desired result is obtained.
\end{proof}

\begin{lemma}
Let $a\in R$, $(af(x),ag(x),h(x))\in Um_3(R[x])$. Then
$$(af(x),ag(x),h(x))\sim_{E_3(R[x])}(f(x),g(x),h(x))$$
\end{lemma}

\begin{proof}
Firstly, there exist $f_1(x),f_2(x),f_3(x) \in R[x]$ such that
\begin{equation*}
af(x) \cdot f_1(x)+ag(x)\cdot f_2(x)+h(x)f_3(x)=1,
\end{equation*}
that is,
\begin{equation*}
a(f(x)\cdot f_1(x)+g(x)\cdot f_2(x))+f_3(x)\cdot h(x)=1.
\end{equation*}
Thus $a$ is a unit modulo $h(x)R[x]$, and hence $(f(x),g(x),h(x)) \in Um_3(R[x])$. By Lemma 3.2, the desired result is obtained.
\end{proof}

\begin{lemma} $([10])$
Let $g(x) \in R[x]$ be of degree $n > 0$ such that $g(0)$ is a unit in $R$. Then for any $f(x)\in R[x]$ and $k \geq \deg f(x)- \deg g(x) + 1$, there exists $h_k(x) \in R[x]$ of degree $ < n$ such that $f(x)\equiv x^k\cdot h_k(x) \mod g(x)R[x]$.
\end{lemma}

\begin{lemma} $([4])$
Let $R$ be a commutative local ring and  $f=(f_1,f_2,\cdots,f_n)\in Um_n(R[x])$ , where $n\geq 3$,  $f_1$ is unitary. Then
$$f(x)\sim_{E_n(R[x])}\,f(0)\sim_{E_n(R)}\,(1,0,\cdots,0).$$
\end{lemma}

\begin{theorem}
Let $V$ be a valuation ring and $(f(x),g(x),e(x))\in Um_3(V[x])$. Then
$$(f(x),g(x),e(x))\sim_{E_3(V[x])}(1,0,0)$$
\end{theorem}

\begin{proof}
Note that $(f(0),h(0),e(0))\in Um_3(V[x])$. Since $V$ is a valuation ring and hence $V$ is local, we may assume that $e(0)=1$.

We proceed by induction on $\deg e(x)$. If $\deg e(x)=0$, the result is obviously true. Now assume that the conclusion holds for $e(x)$ with degree $\leq n-1$. In the case that $\deg e(x) = n \geq 1$, write $e(x)= c_n x^n + c_{n-1} x^{n-1} + \cdots + c_1x + 1$. By Lemma 3.5, for any $s \in \mathbb{N}$ which is big enough, there are $h_1(x), h_2(x) \in V[x]$ such that $f(x) \equiv x^{s}h_1(x)$, $g(x)\equiv x^{s}h_2(x)$ mod $e(x)V[x]$, and $\deg h_1(x),\deg h_2(x) \leq n-1$. Then
\begin{equation}\label{2}
  (f(x),g(x),e(x))\sim_{E_3(V[x])}(x^{s}h_1(x),x^{s}h_2(x),e(x)).
\end{equation}

Since $V$ is a valuation ring, We may assume that $h_1(x)=af_1(x),h_2(x)=bg_1(x)$, where $a,b\in V$, $f_1(x)$ and $g_1(x)$ are primitive polynomials, and $\deg f_1(x) = \deg h_1(x)$, $\deg g_1(x) = \deg h_2(x)$. Since $V$ is a valuation ring, one knows that $a \mid b$ or $b \mid a$. Without loss of generality we may assume that $a \mid b$. Set $b=a \cdot c$. Then $h_2(x) = a \cdot cg_1(x)$, for any $k\in \mathbb{N} (k<s)$. By (2) and  Lemma 3.3, 3.4, we have that
\begin{equation}\label{3}
  \begin{split}
& (f(x),g(x),e(x))\sim_{E_3(V[x])}(ax^{s}f_1(x),ax^{s}cg_1(x),e(x))\\
  &\sim_{E_3(V[x])}(x^{s}f_1(x),x^{s}cg_1(x),e(x))\sim_{E_3(V[x])}(x^{s-k}f_1(x),cx^{s-k}g_1(x),e(x)).
  \end{split}
\end{equation}

Since $f_1(x)$ is primitive, we may suppose that $f_1(x)=a_tx^t+\cdots+a_mx^m+\cdots+a_0$, where $a_t\neq0$, $t\leq n-1$, and $a_m$ is a unit. Pick $k$ such that $s-k+t=n$, i.e., $s-k=n-t\geq 1$ (since $t\leq n-1$).

(1) If $a_t\mid c_n$, let $c_n=d\cdot a_t$, and $e_1(x)=e(x)-d\cdot x^{s-k}f_1(x)$. Since $s-k+t=n$, $\deg e_1(x) \leq n-1$, and $e_1(0)=e(0)=1$. By (3), we have
$$(f(x),g(x),e(x))\sim_{E_3(V[x])}(x^{s-k}f_1(x),cx^{s-k}g_1(x),e_1(x)).$$
Since $\deg e_1(x) \leq n-1$, by the induction hypothesis, we have
$$(f(x),g(x),e(x))\sim_{E_3(V[x])}(1,0,0)$$

(2) If $a_t\nmid c_n$, then $c_n\mid a_t$. Set $a_t=d\cdot c_n$. Then $d$ is not a unit. Let $f_2(x)=x^{n-t}f_1(x)-d\cdot e(x)$. we obtain that
\begin{equation*}
f_2(x) = (a_{t-1}-dc_{n-1})x^{n-1}+\cdots+(a_m-dc_{n-t+m})x^{n-t+m}+\cdots+(-d).
\end{equation*}
Since $d$ is not a unit, $a_m-dc_{n-t+m}$ is a unit, namely that the degree of the term where the coefficient is a unit in $f_2(x)$ is greater than that of $f_1(x)$ by one (since $n-t \geq 1$), and $\deg f_2(x)\leq n-1$.  Since $n-t=s-k$, by (3), we have
\begin{equation}\label{4}
(f(x),g(x),e(x))\sim_{E_3(V[x])}(f_2(x),cx^{s-k}g_1(x),e(x)).
\end{equation}

Suppose that $\deg f_2(x)=t_1$. By Lemma 3.3 and (4), we have

$$(f(x),g(x),e(x))\sim_{E_3(V[x])}(x^{n-t_1}f_2(x),cx^{n+s-t_1-k}g_1(x),e(x))$$

Note that $f_2(x)$ is also primitive, and the degree of the term where the coefficient is a unit in $f_2(x)$ is greater than that of $f_1(x)$ by one. If $c_n$ can be divided by the leading coefficient of $f_2(x)$, by case (1), the result is true. Otherwise, by repeating the preceding procedure, we obtain that $f_i(x)$ such that the leading coefficient of $f_i(x)$ ($i \leq n-m+1$) will be a unit, and
$$(f(x),g(x),e(x))\sim_{E_3(V[x])}(f_i(x),cx^pg_1(x),e(x))$$
Since $f_i(x)$ is unitary, by Lemma 3.6,
$$(f(x),g(x),e(x))\sim_{E_3(V[x])}(1,0,0).$$
Thus the conclusion is true.
\end{proof}

From Theorem 3.1, we obtain the following easy result.

\begin{theorem}
The Hermite ring conjecture is true for valuation rings.
\end{theorem}

\begin{proof}
Let $\omega=(f(x),g(x),e(x)) \in V[x]$ be a unimodular row. It is well-known that $V$ is a Hermite ring. By Theorem 3.1, there is $P \in E_3(V[x])$ such that $(f(x),g(x),e(x))\cdot P=(1,0,0)$. Set
$$N=\left(
\begin{array}{ccc}
0&1&0\\
0&0&1\\
\end{array}
\right)
\cdot P^{-1},~\,~\,
A=\left(
\begin{array}{ccc}
f(x)&g(x)&e(x)\\
 &N& \\
\end{array}
\right),
$$
Then $A\cdot P=I_3$. So $A$ is an invertible square matrix, and $\omega$ can be completed to an invertible square matrix $A$. Thus $V[x]$ is also a Hermite ring.
\end{proof}

\section{Special linear Groups for valuation rings}

Throughout this section $R$ is a unitary commutative ring. We first introduce five well-known lemmas.

\begin{lemma} $([4])$
Let  $A_1=\left(
\begin{array}{cc}
0&I_n\\
-I_n&0\\
\end{array}
\right),
A_2=\left(
\begin{array}{cc}
0&-I_n\\
I_n&0\\
\end{array}
\right)\in E_{2n}(R)$,
Then $A_1,A_2 \in E_{2n}(R)$, where $I_n$ is the $n\times n$ identity matrix over $R$.
\end{lemma}

\begin{lemma} $([9])$
Let $a,a',b\in R$ such that $aa'd-bc=1$ for any $c,d\in R$. We have
$$\left(
\begin{array}{ccc}
aa'&b&0\\
c&d&0\\
0&0&1\\
\end{array}
\right)
\equiv
\left(
\begin{array}{ccc}
a&b&0\\
c&d_1&0\\
0&0&1\\
\end{array}
\right)
\cdot
\left(
\begin{array}{ccc}
a'&b&0\\
c&d_2&0\\
0&0&1\\
\end{array}
\right)
\mod E_3(R)
$$
where $d_1=a'd,d_2=ad$.
\end{lemma}

In fact, if we transpose $aa'$ and $b,c,d$, the conclusion still holds. For the convenience of the reader, we give the following lemma.

\begin{lemma} $([7], [9])$
Let $a,a',b \in R$ such that $aa'd-bc=1$ for any $c,d\in R$. Then
$$\left(
\begin{array}{ccc}
d&b&0\\
c&aa'&0\\
0&0&1\\
\end{array}
\right)
\equiv
\left(
\begin{array}{ccc}
d_1&b&0\\
c&a&0\\
0&0&1\\
\end{array}
\right)
\cdot
\left(
\begin{array}{ccc}
d_2&b&0\\
c&a'&0\\
0&0&1\\
\end{array}
\right)
\mod E_3(R)
$$
where $d_1=a'd,d_2=ad$.
\end{lemma}

\begin{lemma} $([10])$\quad Let  $b,y\in R$. If $b\cdot y$ is nilpotent, then $1\in \langle b,y\rangle \Leftrightarrow 1\in \langle b+y\rangle$.
\end{lemma}

\begin{lemma} $([5])$\quad Let $R$ be a commutative local ring, and let
$$M=
\left(
\begin{array}{ccc}
f(x)&g(x)&0\\
p(x)&q(x)&0\\
0&0&1\\
\end{array}
\right)
\in \SL_3(R[x])$$
Assume that there is at least one monic polynomial among $f(x),g(x),p(x),q(x)$. Then $M \in E_3(R[x])$.
\end{lemma}

Let $V$ be a valuation ring, and
$$M=\left(
      \begin{array}{ccc}
        f(x) & g(x) & 0 \\
        p(x) & q(x) & 0 \\
        0 & 0 & 1 \\
      \end{array}
    \right)\in \SL_3(V[x])
 $$
Observe that $f(0)q(0)-g(0)p(0)=1$. since $V$ is local, we may assume that $g(0)$ and $p(0)$ are units (in this case, $f(0)$ or $q(0)$ is not a unit). When $\deg g(x) \leq \deg p(x)$ (or $\deg p(x) \leq \deg g(x)$), we say that $g(x)$ (or $p(x)$) is a minimal entry of $M$, denoted by $M(m)$. When $\deg M(m) = 0$, $M(m)$ is a unit, and it is obvious that $M \in E_3(V[x])$. By lemma 4.1, we may assume that $g(x)$ is a minimal entry of $M$. Let
$$\begin{aligned}
\SL_3(V[x])_n \triangleq \{&M\in \SL_3(V[x])\mid \exists M_1,\cdots,M_t\ s.t.\ \deg M_i(m)\leq n \\
 &i=1,\cdots,t,\ and\ M\equiv M_1\cdots M_t \mod \ E_3(V[x])\}
 \end{aligned}$$
Obviously, $\SL_3(V[x])_n$ is a multiplicative semigroup, $\SL_3(V[x])_0 = E_3(V[x])$, and $E_3(V[x])\subseteq \SL_3(V[x])_n$ for any positive integer $n$.

\begin{lemma}
Let $V$ be a valuation ring, $M(m)=g(x)$, and
 $$M=\left(
       \begin{array}{ccc}
         f(x) & g(x) & 0 \\
         p(x) & q(x) & 0 \\
         0 & 0 & 1 \\
       \end{array}
     \right)
 \in \SL_3(V[x])_n,$$
 where $f(x)$ is primitive and $\deg f(x) < \deg g(x)$. Then $M \in \SL_3(V[x])_{n-1}$.
 \end{lemma}

\begin{proof}
Let
  $$f(x)=a_lx^l+\cdots+a_mx^m+\cdots+a_0$$
  $$g(x)=b_nx^n+b_{n-1}x^{n-1}+\cdots+b_0$$
where $a_l\neq0$, $b_n\neq0$, $l<n$, $a_m$ and $b_0$ are units.

(1) Suppose that $a_l \mid b_n$. Let $b_n=c \cdot a_l$ and $g_1(x)=g(x)-cx^{n-l}f(x)$. Then $g_1(0)=g(0)$ is a unit, and $\deg g_1(x) \leq n-1$. By Lemma 3.5, for any $s \geq \deg q(x)- \deg g(x)+1$, there is $q_1(x)$ such that $q(x) \equiv x^sq_1(x) \mod g(x)V[x]$, and $\deg q_1(x) \leq n-1$.

We set $s \geq n-l$. By Lemmas 4.3 and 4.5,
  $$\begin{aligned}
  M&\equiv\left(
                  \begin{array}{ccc}
                    f(x) & g(x) & 0 \\
                    p_1(x) & x^sq_1(x) & 0 \\
                    0 & 0 & 1 \\
                  \end{array}
                \right)=\left(
                  \begin{array}{ccc}
                    f(x) & g(x) & 0 \\
                    p_1(x) & x^{n-l}\cdot x^{s-n+l}q_1(x) & 0 \\
                    0 & 0 & 1 \\
                  \end{array}
                \right)\\
                &\equiv \left(
                                \begin{array}{ccc}
                                  x^{n-l}f(x) & g(x) & 0 \\
                                  p_1(x) & x^{s-n+l}q_1(x) & 0 \\
                                  0 & 0 & 1 \\
                                \end{array}
                              \right)\cdot\left(
                                            \begin{array}{ccc}
                                              x^{s-n+l}q_1(x)f(x) & g(x) & 0 \\
                                              p_1(x) & x^{n-l} & 0 \\
                                              0 & 0 & 1 \\
                                            \end{array}
                                          \right)\\
                                          &\equiv\left(
                                                         \begin{array}{ccc}
                                                           x^{n-l}f(x) & g(x) & 0 \\
                                                           p_1(x) & x^{s-n+l}q_1(x) & 0 \\
                                                           0 & 0 & q \\
                                                         \end{array}
                                                       \right)\equiv\left(
                                                                      \begin{array}{ccc}
                                                                        x^{n-l}f(x) & g_1(x) & 0 \\
                                                                        p_1(x) & q_2(x) & 0 \\
                                                                        0 & 0 & 1 \\
                                                                      \end{array}
                                                                    \right)
                                                                    \\&\triangleq M_1\ \mod\ (V[x])
                \end{aligned}$$
where $q_2(x)=x^{s-n+1}q_1(x)-cp_1(x)$. Obviously, $p_1(0)$ is necessarily a unit and $\deg M_1(m) \leq \deg g_1(x) \leq n-1$, so $M_1 \in \SL_3(V[x])_{n-1}$ and $M \in \SL_3(V[x])_{n-1}$.

(2) If $a_l \nmid b_n$, then $b_n \mid a_l$ and $a_l=b_n\cdot c$, where $c$ is not a unit. Let $f_1(x)=x^{n-l}f(x)-c\cdot g(x)$. Then
  $$\begin{aligned}
  M&\equiv\left(
             \begin{array}{ccc}
               x^{n-l}f(x) & g(x) & 0 \\
               p_1(x) & x^{s-n+l}q_1(x) & 0 \\
               0 & 0 & 1 \\
             \end{array}
           \right)\\
  &\equiv \left(
             \begin{array}{ccc}
               f_1(x) & g(x) & 0 \\
               p_2(x) & x^{s-n+l}q_1(x) & 0 \\
               0 & 0 & 1 \\
             \end{array}
           \right)\mod E_3(V[x])
           \end{aligned}
  $$
 Where $f_1(x)=(a_{l-1}-cb_{n-1})x^{n-1}+\cdots+(a_m-cb_{m+n-l})x^{m+n-l}+\cdots+(-cb_0)$. Since $c$ is not a unit, $a_m-cb_{m+n-l}$ is a unit, namely the degree of the term where the coefficient is a unit in $f_1(x)$ is greater than that of $f(x)$ by one (since $n-l \geq 1$), $\deg f_1(x) < n$, and $f_1(x)$ is also primitive. If $b_n$ can be divided by the leading coefficient of $f_1(x)$, then by case (1), $M\in \SL_3(V[x])_{n-1}$. Otherwise, by repeating the preceding procedure, we obtain $f_i(x)$ such that the leading coefficient of $f_i(x)$ ($i\leq n-m$) will be a unit, and
 $$M\equiv\left(
            \begin{array}{ccc}
              f_1(x) & g(x) & 0 \\
              p_2(x) & x^{s-n+l}q_1(x) & 0 \\
              0 & 0 & 1 \\
            \end{array}
          \right)\equiv\left(
                         \begin{array}{ccc}
                           f_i(x) & g(x) & 0 \\
                        \widetilde{p(x)}& \widetilde{q(x)} & 0 \\
                           0 & 0 & 1 \\
                         \end{array}
                       \right) \mod E_3(V[x])
 $$
By lemma 4.5, $M\in E_3(R[x])$, and certainly $M\in \SL_3(V[x])_{n-1}$.
\end{proof}

\begin{lemma}
Let $V$ be a valuation ring, $M(m)=g(x)$, and
  $$M=\left(
        \begin{array}{ccc}
          a & g(x) & 0 \\
          p(x) & q(x) & 0 \\
          0 & 0 & 1 \\
        \end{array}
      \right)\in \SL_3(V[x])_n,
  $$
  where $a\in V$. Then $M\in \SL_3(V[x])_{n-1}$
\end{lemma}

\begin{proof}
(1) If $a=0$ or $a$ is a unit, the result is obviously true.

(2) If $a$ is nilpotent, then by lemma 4.4, $1\in \langle g(x)+a \rangle$, so there exists $f(x)$ such that $f(x)\cdot (g(x)+a)=1$, and
$$\begin{aligned}M&=\left(
                    \begin{array}{ccc}
                      a & g(x) & 0 \\
                      p(x) & q(x) & 0 \\
                      0 & 0 & 1 \\
                    \end{array}
                  \right)
\equiv \left(
            \begin{array}{ccc}
              a & g(x)+a & 0 \\
              p(x) & q_1(x) & 0 \\
              0 & 0 & 1 \\
            \end{array}
          \right)\\
&\equiv\left(
      \begin{array}{ccc}
        1+a & g(x)+a & 0 \\
        p(x)+f(x)q_1(x) & q_1(x) & 0 \\
        0 & 0 & 1 \\
        \end{array}
         \right)\\&\triangleq M_1 \mod E_3(V[x])
\end{aligned}$$
Since $1+a$ is a unit, $M_1\in E_3(V[x])$, and hence $M\in E_3(V[x])\subseteq \SL_3(V[x])_{n-1}$.

 (3) If $a$ is neither a unit nor nilpotent,  we set

 $$
  \begin{aligned}
  &g(x)=b_nx^n+b_{n-1}x^{n-1}+\cdots+b_0\\
 &p(x)=a_mx^m+a_{m-1}x^{m-1}+\cdots+a_0\\
  &q(x)=c_sx^s+c_{s-1}x^{s-1}+\cdots+c_0
  \end{aligned}$$
where  $b_n\neq0$, $a_m\neq0$, $c_s\neq0$. Furthermore, by elementary transformation, we may assume that $a\nmid b_n,a_m$. By virtue of valuation rings, there exist $d_1,d_2\in V$ such that $a=a_m\cdot d_1=b_n\cdot d_2$ and $d_1,d_2$ are not units. If $a_m\cdot b_n=0$, then $a^2=a_md_1b_nd_2=a_mb_nd_1d_2=0$, so $a$ is nilpotent, a contradiction. Therefore, we can assume that $a_m\cdot b_n\neq0$. Since $aq(x)-p(x)g(x)=1$, $s\geq m+n$, and
  $$a\cdot c_s=a\cdot c_{s-1}=\cdots=a\cdot c_{m+n+1}=0,~\,~\,a\cdot c_{m+n}=a_mb_n.$$
Obviously, $ c_s, c_{s-1},\cdots, c_{m+n+1}$ are not units. 

For any $k$, $i$ ($m+n+1\leq i\leq s$), it is necessary that $a^k\mid c_i$. Otherwise, there exists $p_i\in V$ such that $a^{k}=c_i\cdot p_i$, and $a^{k+1}=a\cdot a^k=a\cdot c_i\cdot p_i=0$, a contradiction.

We show that $b_n\nmid c_{m+n}$. If this is not true, then $b_n\mid c_{m+n}$, $c_{m+n}=b_nd_3$, $a_mb_n=ac_{m+n}=a_md_1b_nd_3=a_mb_nd_1d_3$, $a_mb_n(1-d_1d_3)=0$. Since $d_1$ is not a unit, then $1-d_1d_3$ is a unit, and $a_mb_n=0$, a contradiction. Thus $b_n\nmid c_{m+n}$, $c_{m+n}\mid b_n$, and we assume that $b_n=c_{m+n}\cdot d$, where $d\in V$ is not a unit.

Now, we show that there exists $u_1(x)\in V(x)$ such that $h_1(x)=q(x)-u_1(x)\cdot g(x)$ satisfying that $\deg h_1(x)=m+n$, and the leading coefficient of $h_1(x)$ is $c_{m+n}\cdot\varepsilon$, where $\varepsilon\in V$ is a unit. Because $b_n\mid a$, $a^k\mid c_i$ for any positive integer $k$, and $i=m+n+1,\cdots,s$, then $b^k_n\mid c_s$, and there exists $e_s\in V$ such that $c_s=b_n\cdot e_s$, $e_s$ is not a unit. Similarly, for any $k$, $a^k\mid e_s$ since otherwise $a^k=e_s\cdot d_4$, and
\begin{equation*}
a^{k+2}=a\cdot a\cdot a^k=a\cdot b_n\cdot d_2\cdot e_s\cdot d_4=a\cdot (b_n\cdot e_s)\cdot d_2\cdot d_4=a\cdot c_s\cdot d_2\cdot d_4=0
\end{equation*}
so $a$ is nilpotent, a contradiction.

Let $\overline{h_1}(x)=q(x)-e_sx^{s-n}\cdot g(x)$. Then $\deg\overline{h_1}(x)\leq s-1$, and the coefficient $c'_{s-i}$ of $x^{s-i}$ in $\overline{h_1}(x)$ satisfies that $a^k\mid c'_{s-i}$ (since $c'_{s-i}=c_{s-i}-e_s\cdot b_{n-i}$, and $a^k\mid c_{s-i},e_s$), where $i=1,2,\cdots,s-m-n-1$.
\begin{equation*}
c'_{m+n}=c_{m+n}-e_s\cdot b_{m+2n-s}=c_{m+n}(1- b_{m+2n-s}\cdot e_s/c_{m+n}).
\end{equation*}
Since $c_{m+n}\mid b_n$, $b_n\mid a$, and $a^k\mid e_s$, then $c^k_{m+n}\mid e_s$. So $e_s/c_{m+n}$ is not a unit, and ($1- b_{m+2n-s}\cdot e_s/c_{m+n}$) is a unit. Since $b_n\mid a, a^k\mid c'_{k-1}$, then $b^k_n\mid c'_{k-1}$, by repeating the process above, we obtain that
\begin{equation*}
h_1(x)\triangleq q(x)-u_1(x)g(x)=c_{m+n}\cdot\varepsilon \cdot x^{m+n}+\cdots+c_0,
\end{equation*}
where $\varepsilon$ is a unit and $c_0\in V$.

Let $h_2(x)\triangleq h_1(x)-c_0b_0^{-1}g(x)$, then
\begin{equation*}
h_2(x)=x(c_{m+n}\cdot\varepsilon x^{m+n-1}+\cdots+c^{(1)}_0).
\end{equation*}
Let
\begin{equation*}
h_3(x)\triangleq h_2(x)-c^{'}_0b^{-1}_0\cdot x\cdot g(x)=h_1(x)-(c_0b^{-1}_0+c^{(1)}_0b^{-1}_0x)g(x),
\end{equation*}
so
\begin{equation*}
h_3(x)=x^{2}(c_{m+n}\cdot \varepsilon x^{m+n-2}+\cdots+c^{(2)}_0).
\end{equation*}
By repeating the process above $m+1$ times, we obtain that
\begin{equation*}
h(x) = h_1(x) - u_2(x)g(x) = x^{m+1}((c_{m+n} \cdot \varepsilon + b_n \cdot q) x^{n-1} + \cdots+c^{(m+1)}_0),
\end{equation*}
where $q=-c^{(m)}_0b^{-1}_0,c^{(1)}_0,c^{(2)}_0,\cdots,c^{(m+1)}_0\in V$. Since
\begin{equation*}
c_{m+n}\varepsilon+b_nq=c_{m+n}\varepsilon+c_{m+n}dq=c_{m+n}(\varepsilon+dq),
\end{equation*}
where $\varepsilon$ is a unit, and $d$ is not a unit. Then $\overline{\varepsilon} \triangleq \varepsilon + d \cdot q$ is a unit. Let $u(x)=u_1(x)+u_2(x)$, $q_1(x)=c_{m+n}\cdot\overline{\varepsilon}x^{n-1}+\cdots+c^{(1)}_0)$.  Then $h(x)=x^{m+1}q_1(x)$, and
\begin{equation*}
h(x) = h_1(x) -u_2(x)g(x)=q(x)-(u_1(x)+u_2(x))g(x)=q(x)-u(x)g(x).
\end{equation*}

Define
\begin{align*}
p_1(x) & \triangleq p(x)-au(x),\\
g_1(x) & \triangleq g(x)-d\cdot \bar{\varepsilon}^{-1}\cdot x\cdot q_1(x),\\
f_1(x) & \triangleq ax^{m+1}-d\cdot \bar{\varepsilon}^{-1}\cdot x\cdot p_1(x).
\end{align*}
Since $b_n=c_{m+n}\cdot d=d\cdot \overline{\varepsilon}^{-1}\cdot c_{m+n}\cdot \overline{\varepsilon}$, we know that  $\deg g_1(x)\leq n-1$ and $g_1(0)=g(0)$ is a unit. Applying elementary transformations and  Lemma 4.3, 4.5,  we have
$$
\begin{aligned}
M&=\left(
      \begin{array}{ccc}
        a & g(x) & 0 \\
        p(x) & q(x) & 0 \\
        0 & 0 & 1 \\
      \end{array}
    \right)\equiv\left(
                   \begin{array}{ccc}
                     a & g(x) & 0 \\
                     p_1(x) & h(x) & 0 \\
                     0 & 0 & 1 \\
                   \end{array}
                 \right)\equiv\left(
                                \begin{array}{ccc}
                                  a & g(x) & 0 \\
                                  p_1(x) & x^{m+1}q_1(x) & 0 \\
                                  0 & 0 & 1 \\
                                \end{array}
                              \right)\\
                              &\equiv\left(
                                             \begin{array}{ccc}
                                               aq_1(x) & g(x) & 0 \\
                                               p_1(x) & x^{m+1} & 0 \\
                                               0 & 0 & 1 \\
                                             \end{array}
                                           \right)\cdot\left(
                                                         \begin{array}{ccc}
                                                           ax^{m+1} & g(x) & 0 \\
                                                           p_1(x) & q_1(x) & 0 \\
                                                           0 & 0 & 1 \\
                                                         \end{array}
                                                       \right)\equiv\left(
                                                         \begin{array}{ccc}
                                                           ax^{m+1} & g(x) & 0 \\
                                                           p_1(x) & q_1(x) & 0 \\
                                                           0 & 0 & 1 \\
                                                         \end{array}
                                                       \right)\\&\equiv\left(
                                                                      \begin{array}{ccc}
                                                                        f_1(x) & g_1(x) & 0 \\
                                                                        p_1(x) & q_1(x) & 0 \\
                                                                        0 & 0 & 1 \\
                                                                      \end{array}
                                                                    \right)\triangleq M_1 \mod E_3(V[x])
                                                                    \end{aligned}.$$
Since $f_1(x)=ax^{m+1}-d\cdot\overline{\varepsilon}^{-1}\cdot x\cdot p_1(x)$, then $f_1(0)=0$, and $p_1(0)$ must be a unit. Thus $\deg M_1(m)\leq \deg g_1(x) \leq n-1$, $M_1\in \SL_3(V[x])_{n-1}$, and hence $M\in \SL_3(V[x])_{n-1}$.
\end{proof}

 \begin{theorem}
 Let $V$ be a valuation ring, and
 $$M=\left(
       \begin{array}{ccc}
         f(x) & g(x) & 0 \\
         p(x) & q(x) & 0 \\
         0 & 0 & 1 \\
       \end{array}
     \right)\in \SL_3(V[x])
 $$
 Then $M\in E_3(V[x])$.
  \end{theorem}

\begin{proof}
We assume that $M(m)=g(x)$ with $M \in \SL_3(V[x])_n$. By Lemma 3.5, for any $s > \deg f(x) - \deg g(x)+1$, there exists $f_1(x)$ such that $f(x)=x^sf_1(x) \mod g(x)V[x]$, and $\deg f_1(x) < \deg g(x)=n$. Furthermore, let $f_1(x)=af_2(x)$, where $a\in V$, $f_2(x)$ is primitive, and $\deg f_2(x)\leq n-1$. By Lemma 4.2, 4.5, we have
$$\begin{aligned}
M&\equiv\left(
           \begin{array}{ccc}
             ax^sf_2(x) & g(x) & 0 \\
             p_1(x) & q(x) & 0 \\
             0 & 0 & 1 \\
           \end{array}
         \right)\equiv\left(
           \begin{array}{ccc}
             ax^s & g(x) & 0 \\
             p_1(x) & f_2(x)q(x) & 0 \\
             0 & 0 & 1 \\
           \end{array}
         \right)\left(
           \begin{array}{ccc}
             f_2(x) & g(x) & 0 \\
             p_1(x) & ax^sq(x) & 0 \\
             0 & 0 & 1 \\
           \end{array}
         \right)\\&\equiv\left(
           \begin{array}{ccc}
             a & g(x) & 0 \\
             p_1(x) & x^sf_2(x)q(x) & 0 \\
             0 & 0 & 1 \\
           \end{array}
         \right)\left(
           \begin{array}{ccc}
             x^s & g(x) & 0 \\
             p_1(x) & af_2(x)q(x) & 0 \\
             0 & 0 & 1 \\
           \end{array}
         \right)\left(
           \begin{array}{ccc}
             f_2(x) & g(x) & 0 \\
             p_1(x) & ax^sq(x) & 0 \\
             0 & 0 & 1 \\
           \end{array}
         \right)\\&\equiv\left(
           \begin{array}{ccc}
             a & g(x) & 0 \\
             p_1(x) & x^sf_2(x)q(x) & 0 \\
             0 & 0 & 1 \\
           \end{array}
         \right)\left(
           \begin{array}{ccc}
             f_2(x) & g(x) & 0 \\
             p_1(x) & ax^sq(x) & 0 \\
             0 & 0 & 1 \\
           \end{array}
         \right) \mod E_3(V[x]))
\end{aligned}$$

Let $$M_1=\left(
           \begin{array}{ccc}
             a & g(x) & 0 \\
             p_1(x) & x^sf_2(x)q(x) & 0 \\
             0 & 0 & 1 \\
           \end{array}
         \right),\quad M_2=\left(
           \begin{array}{ccc}
             f_2(x) & g(x) & 0 \\
             p_1(x) & ax^sq(x) & 0 \\
             0 & 0 & 1 \\
           \end{array}
         \right)
$$
Since $a\in V$, $f_2(x)$ is primitive and $\deg f_2(x) \leq n-1$, by Lemmas 4.6 and 4.7, $M_1,M_2\in \SL_3(V[x])_{n-1}$, then $M\in \SL_3(V[x])_{n-1}$. By repeating the  procedure above, we have that $M\in \SL_3(V[x])_{n-2},\cdots,\SL_3(V[x])_0=E_3(V[x])$.
\end{proof}

\begin{theorem}
Let $V$ ba a valuation ring. Then $\SL_3(V[x])=E_3(V[x])$.
\end{theorem}

\begin{proof}
Firstly, it is obvious that $E_3(V[x])\subseteq \SL_3(V[x])$. For any $A(x)\in \SL_3(V[x])$, we set
 $$A(x)=\left(
          \begin{array}{ccc}
            a_{11}(x) & a_{12}(x) & a_{13}(x) \\
            a_{21}(x) & a_{22}(x) & a_{23}(x) \\
            a_{31}(x) & a_{32}(x) & a_{33}(x) \\
          \end{array}
        \right)
 $$
 Then $\alpha=(a_{31}(x),a_{32}(x),a_{33}(x))$ is a unimodular row in $V^{1\times3}[x]$. By Theorem 3.1, there exists $E\in E_3(V[x])$ such that $\alpha\cdot E=(0,0,1)$, so
 $$A(x)\equiv\left(
               \begin{array}{ccc}
                 f(x) & g(x) & 0 \\
                 p(x) & q(x) & 0 \\
                 0 & 0 & 1 \\
               \end{array}
             \right) \mod E_3(V[x])
 $$
 By Theorem 4.1, $A(x)\in E_3(V[x])$. Thus $\SL_3(V[x])=E_3(V[x])$.
\end{proof}

\begin{theorem}
Let $V$ be a valuation ring. Then $\SL_n(V[x])=E_n(V[x])$ for $n\geq3$. Furthermore, every matrix $\sigma(x) \in \SL_n(V[x])$ is congruent to $\sigma(0)$ modulo $E_n(V[x])$.
\end{theorem}

\begin{proof}
By Theorem 4.2, we easily obtain that $\SL_n(V[x])=E_n(V[x])$ for $n\geq 3$. For an arbitrary $\sigma(x) \in \SL_n(V[x])$, we have that $\beta(x) = \sigma(0)^{-1} \cdot \sigma(x) \in \SL_n(V[x])=E_n(V[x])$, and $\sigma(x)= \sigma(0)\cdot \beta(x)$. Thus $\sigma(x)$ is congruent to $\sigma(0)$ modulo $E_n(V[x])$.
\end{proof}

\begin{lemma} $([5]\mathrm{GL}_n$-Patching Theorem $)$ Let $n \ge 3$ and $\sigma(x)\in \SL_n(R[x],(x))$. If $\sigma_m(x) \in E_n(R_m[x])$ for every maximal ideal $m$ of $R$, then $\sigma(x) \in E_n(R[x])$.
\end{lemma}

For any $\sigma(x)\in \SL_n(R[x])$, obviously, $\beta(x)=\sigma(0)^{-1}\cdot \sigma(x)\in \SL_n(R[x],(x))$. Combining Theorem 4.3 and the $\mathrm{GL}_n$-Patching Theorem, we now can give an answer to question (3).

\begin{theorem}
Let $R$ be an arithmetical ring and $n \geq 3$. Then $\SL_n(R[x]) = \SL_n(R) \cdot E_n(R[x])$.
\end{theorem}

\begin{proof}
Let $\sigma(x) \in \SL_n(R[x])$. Then $\beta(x) = \sigma(0)^{-1}\cdot \sigma(x) \in \SL_n(R[x],(x))$. Since $R_m$ is a valuation ring, by Theorem 4.3, $\beta_m(x)\in E_n(R_m[x])$ for every maximal ideal $m$ of $R$, by the virtue of $\mathrm{GL}_n$-Patching Theorem, $\beta(x) \in E_n(R[x])$. Thus $\sigma(x) = \sigma(0)\cdot \beta(x) \in \SL_n(R) \cdot E_n(R[x])$, and hence $\SL_n(R[x]) = \SL_n(R) \cdot E_n(R[x])$.
\end{proof}

From Theorem 4.4, we can obtain the following result directly.

\begin{theorem}
Let $R$ be an arithmetical ring and $n\ge 3$. Then every matrix $\sigma(x)\in \SL_n(R[x])$ is congruent to $\sigma(0)$ modulo $E_n(R[x])$.
\end{theorem}

Valuation rings form a class of important local rings. We conclude that Questions (1), (2) and (3) are completely solved by Theorems 3.2, 4.1,4.3 and 4.5.

\renewcommand\refname
\end{document}